\documentclass{article}
\usepackage{amsrefs}
\usepackage{amssymb}
\usepackage{amsmath}
\usepackage{amsthm}
\usepackage{graphicx}

\DeclareMathOperator{\Supp}{Supp}

\DeclareMathOperator{\sepgon}{sepgon}
 \numberwithin{equation}{section}
\DeclareMathOperator{\Ext}{Ext} \DeclareMathOperator{\Proj}{Proj}
\newtheorem{theorem}{Theorem}[section]
\newtheorem{definition}[theorem]{Definition}

\newtheorem{proposition}[theorem]{Proposition}
\newtheorem{lemma}[theorem]{Lemma}
\newtheorem{definition/construction}[theorem]{Definition/Construction}

\newtheorem{corollary}[theorem]{Corollary}
\theoremstyle{remark}
\newtheorem*{remark}{Remark}
\newtheorem*{problem}{Problem}
\newtheorem*{subject}{2000 Mathematics Subject Classification}
\newtheorem*{keywords}{Keywords}

\author{Marc Coppens\footnote{Katholieke Hogeschool Kempen, Departement IBW,
Kleinhoefstraat 4, B-2440 Geel, Belgium; K.U.Leuven, Departement of Mathematics,
Celestijnenlaan 200B, B-3001 Leuven, Belgium; email: marc.coppens@khk.be}}
\title{Pencils on separating $(M-2)$-curves}
\date{}

\begin{document}
\maketitle \noindent

\begin{abstract}

A separating ($M-2$)-curve is a smooth geometrically irreducible
real projective curve $X$ such that $X(\mathbb{R})$ has $g-1$
connected components and $X(\mathbb{C})\setminus X(\mathbb{R})$ is
disconnected. Let $T_g$ be a Teichm\"uller space of separating
($M-2$)-curves of genus $g$. We consider two partitions of $T_g$,
one by means of a concept of special type, the other one by means of
the separating gonality. We show that those two partitions are very
closely related to each other. As an application we obtain the
existence of real curves having isolated real linear systems
$g^1_{g-1}$ for all $g\geq 4$.

\end{abstract}

\begin{subject}
14H05; 14H51; 14P99
\end{subject}

\begin{keywords}
real curve, linear pencil, separating gonality, special type,
Teichm\"uller space
\end{keywords}

\section{Introduction}\label{section1}

Let $X$ be a smooth real projective curve of genus $g$. We assume
$X$ is complete and geometrically irreducible, hence the set
$X(\mathbb{C})$ of complex points is in a natural way a compact
Riemann surface of genus $g$. Let $X(\mathbb{R})$ be the set of real
points and assume it is not empty. Let $C_1, \cdots, C_s$ be the
connected components of $X(\mathbb{R})$. It is well-known that
$s\leq g+1$ (Harnack's inequality). Let $f:X\rightarrow
\mathbb{P}^1$ be a morphism of degree $k$. It is known that the
parity of the fibers (counted with multiplicities) of $f|_{C_i}:C_i
\rightarrow \mathbb{P}^1(\mathbb{R})$ is constant. In particular in
case this parity is odd then $f(C_i)=\mathbb{P}^1(\mathbb{R})$. In
our paper \cite{ref1} we considered the following problem.

\begin{problem} Fix $k$, $s'\leq s$ and $s'$ components $C_{i_1}, \cdots,
C_{i_{s'}}$ of $X(\mathbb{R})$. Does there exist a morphism
$f:X\rightarrow \mathbb{P}^1$ of degree $k$ such that $f$ has odd
parity on $C_j$ for $j\in \{ i_1, \cdots, i_{s'} \}$ and $f(C_j)\neq
\mathbb{P}^1(\mathbb{R})$ for $j \notin \{ i_1, \cdots, i_{s'} \}$.
\end{problem}

Of course $s-s' \equiv 0 \pmod{2}$ is a necessary condition and in
\cite{ref1}*{Proposition 1} it is proved that in case $k=g+1$ this
condition is also sufficient. However in case $k=g$ then this
condition is not sufficient because of the following example
mentioned in \cite{ref1}*{Example 3}. A real curve $X$ of genus 3
with $s=2$ and such that $X(\mathbb{R})$ disconnects $X(\mathbb{C})$
is isomorphic to a smooth plane real curve of degree 4 having two
nested ovals ($C_1$ in the inner part of $C_2$). Taking $k=3$,
$s'=1$ and $i_1=1$, then for each morphism $f:X\rightarrow
\mathbb{P}^1$ of degree 3 having odd parity on $C_1$ one has
$f(C_2)=\mathbb{P}^1(\mathbb{R})$.

A real curve $X$ such that $X(\mathbb{R})$ disconnects
$X(\mathbb{C})$ is called separating and it is shown in
\cite{ref1}*{Theorem 1.A} that the condition $s-s' \equiv 0
\pmod{2}$ is sufficient for an affirmative answer to the problem in
case $k=g$ and $X$ is not separating. In \cite{ref2}*{Example 5.9}
as a second example one finds separating curves of genus 4 with
$s=3$ such that there exist components $C_1$ and $C_2$ of
$X(\mathbb{R})$ such that for each morphism $f:X\rightarrow
\mathbb{P}^1$ of degree 4 having odd parity on $C_1$ and $C_2$ one
has $f(C_3)=\mathbb{P}^1(\mathbb{R})$ ($C_3$ is the other component
of $X(\mathbb{R})$ different from $C_1$ and $C_2$). The argument
makes use of the description of a canonically embedded curve of
genus 4 in $\mathbb{P}^3$ as the intersection of a cubic and a
quadric surface. In both examples we have $s=g-1$. Classically, a
real curve $X$ satisfying $s=g+1$ is called an $M$-curve and in the
literature a real curve satisfying $s=g+1-a$ is also called an
($M-a$)-curve. So both examples are separating ($M-2$)-curves. In
Theorem \ref{SpecialType} we prove that for all $g\geq 3$ there
exists a separating ($M-2$)-curve $X$ having components $C_1,
\cdots, C_{g-1}$ of $X(\mathbb{R})$ such that, if $f:X\rightarrow
\mathbb{P}^1$ is a morphism of degree $g$ having odd parity on $C_2,
\cdots, C_{g-1}$ then $f(C_1)=\mathbb{P}^1(\mathbb{R})$ (in this
statement the numbering of the components of $X(\mathbb{R})$ is
important). We say such a curve is of special type. Theorem
\ref{SpecialType} is a direct consequence of Proposition
\ref{Proposition1}. In Proposition \ref{Proposition1} we prove a
more geometric statement related to this concept: the existence of a
canonically embedded separating ($M-2$)-curve $X$ possessing a
strong kind of linking between the connected components of
$X(\mathbb{R})$.

We prove a stronger statement. Let $T_g$ be the Teichm\"uller space
parameterizing separating ($M-2$)-curves of genus $g$. In case $t\in
T_g$ then we write $X_t$ to denote the corresponding real curve.
This space $T_g$ is a real connected manifold of dimension $3g-3$.
We say a property $P$ holds for a general separating ($M-2$)-curve
if there exists a non-empty open subset $U$ of $T_g$ such that $P$
holds for all curves $X_t$ with $t\in U$ (roughly speaking: the
curves satisfying property $P$ have the maximal $3g-3$ moduli). From
Corollary \ref{general} it follows that for $g\geq 4$ both
properties ''being of simple type'' and ''not being of simple type''
do hold for a general separating ($M-2$)-curve of genus $g$ (in case
$g=3$ all separating ($M-2$)-curves are of special type). Let
$T_{g,s}$ (resp. $T_{g,ns}$) be the set of points $t\in T_g$ such
that $X_t$ is of special type (resp. $X_t$ is not of special type).
So we have a partition $T_g=T_{g,s}\cup T_{g,ns}$. In Lemma
\ref{SpecialTypeClosed} we show $T_{g,s}$ is closed, hence
$T_{g,ns}$ is open. This partition turns out to be closely related
to another very natural parition of $T_g$.

In case a real curve $X$ has a morphism $f:X\rightarrow
\mathbb{P}^1$ with $X(\mathbb{R})=f^{-1}(\mathbb{P}^1(\mathbb{R}))$
then $X$ is separating. Such morphism is called a separating
morphism. In \cite{ref11} we introduce the separating gonality
$\sepgon (X)$ of a separating real curve $X$: it is the minimal
degree such that there exists a separating morphism $f:X\rightarrow
\mathbb{P}^1$. For a separating ($M-2$)-curve $X$ trivially one has
$\sepgon (X)\geq g-1$. On the other hand, from \cite{ref6} it
follows $\sepgon (X)\leq g$ and in \cite{ref11} it is proved that
both possibilities $g-1$ and $g$ do occur. Let $T_{g,g}$ (resp.
$T_{g,g-1}$) be the set of points $t\in T_g$ such that $\sepgon
(X_t)=g$ (resp. $\sepgon (X_t)=g-1$). So we obtain a second
partition $T_g=T_{g,g} \cup T_{g,g-1}$ and the relation between both
partitions is given by the fact that the closure
$\overline{T_{g,ns}}$ of $T_{g,ns}$ is equal to $T_{g,g-1}$ (see
Corollary \ref{relatie3}). It follows that $T_{g,g}=T_{g,s}\setminus
(T_{g,s}\cap \overline{T_{g,ns}})$ is a non-empty open subset of
$T_g$.

The fibers of a separating morphism $f:X\rightarrow \mathbb{P}^1$ of
degree $g-1$ correspond to a linear system $g^1_{g-1}$ on $X$.
Complete linear systems of degree $g-1$ and dimension at least one
on $X$ are parameterized by a subscheme $W^1_{g-1}$ of the Jacobian
$J(X)$ and in case $X$ is not hyperelliptic then all components of
$W^1_{g-1}(\mathbb{C})$ have dimension $g-4$. Linear systems
$g^1_{g-1}$ corresponding to separating morphisms of degree $g-1$ on
a separating ($M-2$)-curve $X$ are parameterized by a dense open
subset of some irreducible components of $W^1_{g-1}(\mathbb{R})$. In
case $X$ is a general non-special separating ($M-2$)-curve then all
such components have real dimension $g-4$. If $X$ is a special
separating ($M-2$)-curve with $\sepgon (X)=g-1$ then our results
imply $X=X_t$ for some $t\in T_{g,s}\cap \overline{T_{g,ns}}$. In
Corollary \ref{ExistenceSpecialSepgon} we prove this intersection is
non-empty and in Proposition \ref{Lemma1} we prove such $X$ has
finitely many $g^1_{g-1}$ associated to separated morphisms of
degree $g-1$. In particular for such curve $W^1_{g-1}(\mathbb{R})$
has isolated points (see Corollary \ref{corollary1}). In case $g\geq
5$ this is remarkable when compared to $\dim
(W^1_{g-1}(\mathbb{C}))=g-4$. The finiteness follows from the
following remarkable fact proved in Proposition \ref{Lemma1}. If $X$
is an ($M-2$)-curve of special type then a linear system $g^1_{g-1}$
on $X$ corresponding to a separated morphism $f:X\rightarrow
\mathbb{P}^1$ is half-canonical.

\section{Preliminaries and notations}\label{section2}

A \emph{real curve} $X$ is a one-dimensional geometrically connected
projective variety defined over the field $\mathbb{R}$ of the real
numbers. Using a base extension $\mathbb{R}\subset \mathbb{C}$ we
obtain a complex curve $X_{\mathbb{C}}$. Its set of closed points is
denoted by $X(\mathbb{C})$ and it is called the space of complex
points on $X$. Complex conjugation related to $\mathbb{R}\subset
\mathbb{C}$ defines a complex conjugation on $X(\mathbb{C})$, for
$P\in X(\mathbb{C})$ we write $\overline{P}$ to denote the complex
conjugated point. On $X$ itself (considered as a scheme) there are
two types of closed points according to the residu field being
$\mathbb{R}$ or $\mathbb{C}$. In case the residu field is
$\mathbb{R}$ then we say it is a \emph{real point} on $X$. The set
of real points is denoted by $X(\mathbb{R})$ and there exists a
natural inclusion $X(\mathbb{R})\subset X(\mathbb{C})$. In case the
residu field is $\mathbb{C}$ then the closed point on $X$
corresponds to two conjugated points $P$, $\overline{P}$ on
$X(\mathbb{C})\setminus X(\mathbb{R})$. Such closed point on $X$ is
denoted by $P+\overline{P}$ and it is called a \emph{non-real point}
on $X$. The real projective line $\Proj (\mathbb{R}[X_0,X_1])$ is
denoted by $\mathbb{P}^1$. A linear system of dimension $r$ and
degree $d$ on a smooth real curve $X$ is denoted by $g^r_d$. It is a
projective space of linearly equivalent real divisors on $X$.

In case $X_{\mathbb{C}}$ is a smooth (resp. stable) complex curve we
call $X$ a smooth (resp. stable) real curve. The moduli functor of
stable curves of genus $g$ is not representable, hence there is no
universal family. Instead we make use of so-called suited families
of stable curves.

\begin{definition}
Let $X$ be a real stable curve of genus $g$. A \emph{suited family
of stable curves of genus $g$ for $X$} is a projective morphism $\pi
: \mathcal{C} \rightarrow S$ defined over $\mathbb{R}$ such that

\begin{enumerate}
\item
$S$ is smooth, geometrically irreducible and quasi-projective.
\item
Each geometric fiber of $\pi$ is a stable curve of genus $g$.
\item
For each $s\in S(\mathbb{C})$ the Kodaira-Spencer map
$T_s(S)\rightarrow \Ext ^1(\Omega _{\pi ^{-1}(s)},\mathcal{O}_{\pi
^{-1}(s)})$ is surjective (here $\Omega _{\pi ^{-1}(s)}$ is the
sheaf of K\"ahler differentials).
\item
There exists $s_0 \in S(\mathbb{R})$ such that $\pi ^{-1}(s_0)\cong
X$ over $\mathbb{R}$.
\end{enumerate}

In case $X$ is smooth we also assume $\pi $ is a smooth morphism.

\end{definition}

In \cite{ref11}*{Lemma 4} it is explained such suited families do
exist. Let $X$ be a smooth real curve and let $\pi :
\mathcal{C}\rightarrow S$ be a suited family for $X$. Let $k\in
\mathbb{Z}$ with $k\geq 2$. There exists a quasi-projective morphism
$\pi _k: H_k(\pi ) \rightarrow S$ representing morphisms of degree
$k$ from fibers of $\pi$ to $\mathbb{P}^1$ (see
\cite{ref12}*{Section 4.c}). Let $f:X \rightarrow \mathbb{P}^1$ be a
morphism of degree $k$. It defines an invertible sheaf
$L=f^*(\mathcal{O}_{\mathbb{P}^1}(1))$ of degree $k$ on $X$. The
morphism $f$ induces an exact sequence $0 \rightarrow T_X
\rightarrow f^*(T_{\mathbb{P}^1}) \rightarrow N_f \rightarrow 0$
($N_f$ is defined by this exact sequence) and since
$T_{\mathbb{P}^1}\cong \mathcal{O}_{\mathbb{P}^1}(2)$ this exact
sequence looks like

\begin{equation*}
0 \rightarrow T_X \rightarrow L^{\otimes 2} \rightarrow N_f
\rightarrow 0
\end{equation*}

The morphism $f$ corresponds to a point $[f]$ on $H_k(\pi )$ and
from Horikawa's deformation theory of holomorphic maps (see
\cite{ref13}, see also \cite{ref14}*{3.4.2}), it follows
$T_{[f]}(H_k(\pi ))$ is canonically identified with $H^0(X,N_f)$ and
since $H^1(X,N_f)=0$ it follows $H_k(\pi )$ is smooth of dimension
$2k+2g-2$. Moreover $T_{s_0}(S)$ is isomorphic to $H^1(X,T_X)$ and
the connecting homomorphism $H^0(X,N_f)\rightarrow H^1(X,T_X)$
associated to the exact sequence is identified with the tangent map
$d_{[f]}(\pi _k):T_{[f]}(H_k(\pi ))\rightarrow T_{s_0}(S)$. In
particular $d_{[f]}(\pi _k)$ is surjective in case $H^1(X,L^{\otimes
2})=0$. Hence the condition $H^1(X,L^{\otimes 2})=0$ implies $\pi
_k^{-1}(s_0)$ has dimension $2k-g+1$ and it is smooth at $[f]$. In
\cite{ref7} we introduced the topological degree of $f$. Choose an
orientation on $\mathbb{P}^1(\mathbb{R})$. For each component $C$ of
$X(\mathbb{R})$ (this is a smooth real manifold diffeomorphic to
$S^1$) we consider the restriction $f| _C : C \rightarrow
\mathbb{P}^1(\mathbb{R})$ and we fix an orientation on $C$ such that
$\deg (f| _C)\geq 0$. We say $f$ has \emph{of topological degree
$(d_1, \cdots , d_s)$} with $d_1 \geq \cdots \geq d_s \geq 0$ if
there is a numbering $C_1, \cdots , C_s$ of all components of
$X(\mathbb{R})$ such that $\deg (f| _{C_i})=d_i$. In families of
morphisms from smooth real curves to $\mathbb{P}^1$ this topological
degree is constant, hence it is constant on connected components of
$H_k(\pi )(\mathbb{R})$.

Let $X$ be a smooth real curve. In case $X(\mathbb{R})\neq
\emptyset$ then it is a disjoint union of $s=s(X)$ connected
components diffeomorphic to a circle. In case
$X(\mathbb{C})\setminus X(\setminus\mathbb{R})$ is not connected it
has two connected components and $X$ is called a \emph{separating
real curve}. For a separating real curve one has $1\leq s\leq g-1$
and $s \equiv g+1 \pmod{2}$. In case $s=g+1-a$ then $X$ is called an
\emph{($M-a$)-curve}. The following definitions are already
mentioned at the introduction.

\begin{definition}
A separating ($M-2$)-curve $X$ is of \emph{special type} if there
exists a component $C$ of $X(\mathbb{R})$ such that for each
morphism $f:X\rightarrow \mathbb{P}^1$ of degree $g$ having odd
parity on each connected component $C'\neq C$ of $X(\mathbb{R})$ one
has $f(C)=\mathbb{P}^1(\mathbb{R})$. If no such component $C$ exists
then we say $X$ is \emph{not of special type}.
\end{definition}

\begin{definition}
A morphism $f:X\rightarrow \mathbb{P}^1$ is called \emph{a
separating morphism} if
$f^{-1}(\mathbb{P}^1(\mathbb{R}))=X(\mathbb{R})$.
\end{definition}

In case $X$ has a separating morphism then $X$ is a separating real
curve.

\begin{definition}
The \emph{separating gonality} $\sepgon (X)$ of a separating real
curve $X$ is the minimal degree $k$ such that there exists a
separating morphism $f:X\rightarrow \mathbb{P}^1$ of degree $k$.
\end{definition}

As already mentioned in the introduction, in case $X$ is a
separating ($M-2$)-curve then $\sepgon (X)$ is either $g$ or $g+1$.
As mentioned in the introduction we write $T_g$ to denote a
Teichm\"uller space parameterizing separating ($M-2$)-curves and we
obtain two partitions $T_g=T_{g,s}\cup T_{g,ns}$ and
$T_g=T_{g,g}\cup T_{g,g-1}$. Remember $T_g$ is a smooth real
manifold of dimension $3g-3$ and it has a universal family
$t_g:\mathcal{X}_g \rightarrow T_g$. For each separating real
($M-2$)-curve $X_0$ there exists $t_0\in T_g$ such that
$t_g^{-1}(t_0)\cong X_0$. Moreover, if $\pi : \mathcal{C}_g
\rightarrow S$ is a suited family of curves for $X_0$ and $s_0\in
S(\mathbb{R})$ with $\pi ^{-1}(s_0)\cong X_0$ then there exist
neighborhoods $U$ (resp. $V$) of $t_0$ (resp. $s_0$) in $T_g$ (resp.
$S(\mathbb{R})$) and a diffeomorphism $U\rightarrow V$ such that, if
$u\in U$ maps to $v\in V$ then $t_g^{-1}(u)\cong \pi ^{-1}(v)$.

\begin{lemma}\label{coverings}
Let $X$ be a separating ($M-2$)-curve, let $C_1, \cdots, C_{g-1}$ be
the connected components of $X(\mathbb{R})$ and assume
$f:X\rightarrow \mathbb{P}^1$ is a covering of degree $g$ having odd
parity on $C_1, \cdots, C_{g-2}$. Then $f(C_{g-1})\neq
\mathbb{P}^1(\mathbb{R})$ unless $f| _{C_{g-1}}$ is an unramified
covering $C_{g-1} \rightarrow \mathbb{P}^1(\mathbb{R})$ of degree 2.
\end{lemma}
\begin{proof}
First of all, the morphism $f$ has even parity on $C_{g-1}$ (because
of the necessary condition involving $s$ and $s'$ for the problem
mentioned in the introduction). Since each fiber above a point $x$
of $\mathbb{P}^1(\mathbb{R})$ contains a point of $C_i$ for $1\leq
i\leq g-2$ it contains at most 2 points of $C_{g-1}$ (counted with
multiplicities) and there cannot be a ramification point on
$C_{g-1}$ of index more than two. If there is a ramification point
$x_0$ on $C_{g-1}$ of index two then close to $f(x_0)$ there exists
$x'\in \mathbb{P}^1(\mathbb{R})$ such that $f^{-1}(x')$ contains a
non-real point. It follows $f^{-1}(x')$ cannot contain a point of
$C_{g-1}$ hence $f(C_{g-1})\neq \mathbb{P}^1(\mathbb{R})$. Hence
$f(C_{g-1})=\mathbb{P}^1(\mathbb{R})$ implies $f$ has no
ramification point on $C_{g-1}$, hence $f| _{C_{g-1}}$ is an
unramified covering $C_{g-1} \rightarrow \mathbb{P}^1(\mathbb{R})$
of degree two.
\end{proof}

\begin{remark}
In the situation of the previous lemma, if $f
(C_{g-1})=\mathbb{P}^1(\mathbb{R})$ it follows
$f^{-1}(\mathbb{P}^1(\mathbb{R}))=X(\mathbb{R})$, hence $f$ is a
separating morphism of degree $g$. In that case $f$ has topological
degree $(2,1, \cdots, 1)$. In case $f(C_{g-1})\neq
\mathbb{P}^1(\mathbb{R})$ it has topological degree $(1, \cdots, 1,
0)$.
\end{remark}

\begin{lemma}\label{SpecialTypeClosed}
$T_{g,ns}\subset T_g$ is open and (hence) $T_{g,s}\subset T_g$ is
closed.
\end{lemma}

\begin{proof}
We are going to prove that $T_{g,ns}\subset T_g$ is open. Let $t\in
T_{g,ns}$ and let $X=t_g^{-1}(t)$. Let $\pi : \mathcal{C}\rightarrow
S$ be a suited family for $X$ and $s\in S(\mathbb{R})$ such that
$\pi ^{-1}(s)\cong X$. It is enough to prove there exists a
classical open neighborhood $U$ of $s$ in $S(\mathbb{R})$ such that
for all $s'\in U$ the curve $\pi ^{-1}(s')$ is a separating
($M-2$)-curve not of special type. It is well-known that points in
$S(\mathbb{R})$ close to $s$ do correspond to separating
($M-2$)-curves, so we only have to show they are also of non-special
type.

Choose a component $C$ of $X$. Since the curve is not of special
type there exists a covering $f:X \rightarrow \mathbb{P}^1$ of
degree $g$ such that it has topological degree $(1, \cdots, 1, 0)$
and $f(C)\neq \mathbb{P}^1(\mathbb{R})$. Consider $\pi _g:H_g(\pi)
\rightarrow S$ with $H_g(\pi )$ parameterizing morphisms of degree
$g$ from fibers of $\pi$ to $\mathbb{P}^1$ and now let $H$ be the
connected component of $H_g(\pi )(\mathbb{R})$ containing $[f]$.
From the deformation theory of Horikawa we know $H$ is smooth of
dimension $4g-2$. Moreover, $f$ corresponds to an invertible sheaf
$L$ of degree $g$, therefore $H^1(X,L^{\otimes 2})=0$, hence the
description of the tangent map of $\pi _g$ at $[f]$ implies this
tangent map has maximal rank. So the image of a neighborhood of
$[f]$ on $H$ contains a neighborhood $U$ of $s$ in $S$. Intersecting
those neighborhoods for all choices of $C$ (again denoted by $U$) we
obtain for each $s'\in U$ and for each component $C'$ of
$\pi^{-1}(s')(\mathbb{R})$ the existence of a morphism $f' : \pi
^{-1}(s')\rightarrow \mathbb{P}^1$ of topological degree $(1,
\cdots, 1, 0)$ having even parity on $C'$, hence $f'(C')\neq
\mathbb{P}^1(\mathbb{R})$ because of Lemma \ref{coverings}. This
means $\pi^{-1}(s')$ is not of special type.
\end{proof}

\begin{lemma}\label{SepgonClosed}
$T_{g,g-1}\subset T_g$ is closed and (hence) $T_{g,g}\subset T_g$ is
open.
\end{lemma}

\begin{proof}

Let $X_0$ be a curve corresponding to a point on the closure of
$T_{g,g-1}$. Then $X_0$ is the limit of a family of separating
($M-2$)-curves $X_t$ ($t>0$) having a separating morphism $f_t:X_t
\rightarrow \mathbb{P}^1$ of degree $g-1$. Since $X_t(\mathbb{R})$
has $g-1$ components such morphism has to be of topological type
$(1, \cdots, 1)$. Therefore the fiber of $f_t$ over a real point of
$\mathbb{P}^1$ is of type $P_1 + \cdots + P_{g-1}$ with $P_i$
belonging to different components of $X_t(\mathbb{R})$. The limit of
such divisor on $X_0$ is of the same type and belongs to a complete
linear system of dimension at least 1. So it defines a complete
linear system $g^r_{g-1}$ for some $r\geq 1$ having odd degree on
each component $C$ of $X_0(\mathbb{R})$. In case $r> 1$ then for
$P_1, P'_1$ on the same component $C$ of $X_0(\mathbb{R})$ there
should exist $D\in g^r_{g-1}$ containing $P_1+P'_1$. Since $D$
should contain a point of each component of $X_0(\mathbb{R})$, this
is impossible. So $r=1$. In case $D$ would have a base point (say
$P_1$) then for $P'_1$ general on the same component there should
exist $D\in g^1_{g-1}$ containing $P_1+P'_1$ giving the same
contradiction. So $g^1_{g-1}$ corresponds to a base point free
linear system having odd degree on each component of
$X_0(\mathbb{R})$, so it defines a separating morphism $f_0:X_0
\rightarrow \mathbb{P}^1$ of degree $g-1$.
\end{proof}

\section{Existence of separating
($M-2$)-curves of special type}\label{section3}

\begin{theorem}\label{SpecialType}
For each $g\geq 3$ there exists a separating (M-2)-curve $X$ of
special type.
\end{theorem}

This theorem is an immediate corollary of the next proposition. This
proposition shows that the components of the real locus of a
canonically embedded real curve can be strongly linked to each
other. Therefore the proposition describes the geometric reason for
the existence of separating ($M-2$)-curves of special type. It would
be interesting to obtain more information concerning the way the
components of the real locus of a canonically embedded real curve
can be linked.

For a curve $X$ embedded in some projective space $\mathbb{P}$ and
an effective divisor $E$ on $X$ we denote $\langle E \rangle$ for
the linear span: it is the intersection of hyperplanes $H$ of
$\mathbb{P}$ such that $H.X \geq E$ (and it is $\mathbb{P}$ in case
such hyperplane does not exist).

\begin{proposition}\label{Proposition1}
For all $g\geq 3$ there is a canonically embedded ($M-2$)-curve $X
\subset \mathbb{P}^{g-1}$ having real components $C_1, \cdots,
C_{g-1}$ of $X(\mathbb{R})$ such that
\begin{enumerate}
\item for all $P_i \in C_i$ ($1\leq i\leq g-1$) one has $\dim
(\langle P_1, \cdots, P_{g-1} \rangle)=g-2$
\item for all $P_i\in C_i$ ($2\leq i\leq g-1$) and for each real
hyperplane $H\subset \mathbb{P}^{g-1}$ containing $\langle P_2,
\cdots, P_{g-1} \rangle$ one has $H\cap C_1\neq \emptyset$.
\end{enumerate}
\end{proposition}

\begin{proof}[Proof of Theorem \ref{SpecialType}]

Let $X$ be as described in Proposition \ref{Proposition1}. Take $P_i
\in C_i$ ($2\leq i\leq g-1$) and consider $|K_X-(P_2+ \cdots +
P_{g-1})|$. From (1) in Proposition \ref{Proposition1} we have $\dim
(\langle P_2, \cdots, P_{g-1} \rangle)=g-3$ hence $\dim (|K_X-(P_2+
\cdots +P_{g-1}|)=1$ ($|K_X - (P_2 + \cdots + P_{g-1})|$ is the
linear system induced by the pencil of hyperplanes in
$\mathbb{P}^{g-1}$ containing $\langle P_2, \cdots , P_{g-1}
\rangle$, it is denoted by $g^1_g$). Since $K_X$ has even degree on
each component of $X(\mathbb{R})$ it follows $g^1_g$ has odd degree
on $C_i$ for $2\leq i\leq g$ and even degree on $C_1$. From (2) in
Proposition \ref{Proposition1} it follows each divisor $D\in g^1_g$
contains some point of $C_1$, hence it contains a divisor of degree
2 with support on $C_1$. This proves each divisor of $g^1_g$ is of
the type $D=P'_1+P''_1+P'_2+ \cdots +P'_{g-1}$ with $P'_i\in C_i$
for $1\leq i\leq g-1$ and $P''_1\in C_1$.

Assume $P'_i$ is a base point of $g^1_g$ for some $2\leq i\leq g-1$
then no divisor of $g^1_g$ can contain another point of $C_i$. This
is impossible hence $P'_i$ is not a base point for $2\leq i\leq
g-1$. Assume e.g. $P''_1$ is a base point for $g^1_g$ then $\dim
|P'_1 + P'_2 + \cdots +P'_{g-1}|=1$. Then the geometric version of
the Riemann-Roch Theorem (see e.g. \cite{ref15}*{p. 248}) implies
$\dim \langle P'_1, \cdots, P'_{g-1} \rangle=g-3$ contradicting (1)
in Proposition \ref{Proposition1}. So $g^1_g$ is base point free and
it defines a covering $f:X\rightarrow \mathbb{P}^1$ having odd
degree on $C_i$ for $2\leq i\leq g-1$ and such that $C_1$ also
dominates $\mathbb{P}^1(\mathbb{R})$. From the description of the
divisors of $g^1_g$ it follows all fibers of $f$ over
$\mathbb{P}^1(\mathbb{R})$ are totally real, hence $X$ is a
separating curve.

Conversely, if $f:X\rightarrow \mathbb{P}^1$ is a morphism of degree
$g$ having odd parity on $C_i$ for $2\leq i\leq g-1$, then for a
real fiber $E$ of $f$ one has $|K_X -E|\neq \emptyset$ and $|K_X -
E|$ has odd parity on $C_2, \cdots, C_{g-1}$. Since $\deg
(K_X-E)=g-2$ each divisor of $|K_X-E|$ is of type $P_2+ \cdots +
P_{g-1}$ with $P_i\in C_i$ for $2\leq i\leq g-1$. So $f$ corresponds
to $|K_X -(P_2+ \cdots + P_{g-1})|$ and we already proved
$f(C_1)=\mathbb{P}^1(\mathbb{R})$. This shows $X$ is of special
type.

\end{proof}

For a curve $X$ satisfying properties (1) and (2) of Proposition
\ref{Proposition1} we found $| K_X-(P_2+ \cdots +P_{g-1})|$ with
$P_i\in C_i$ ($2\leq i\leq g-1$) defines a covering $\pi :
X\rightarrow \mathbb{P}^1$ such that $C_i$ dominates
$\mathbb{P}^1(\mathbb{R})$ for $1\leq i\leq g-1$. In particular
$\pi$ is not ramified at some real point of $X$. Since $\deg (\pi |
_{C_1})=2$, it also implies condition (2) of Proposition
\ref{Proposition1} is equivalent to: for all $P_i\in C_i$ ($2\leq
i\leq g-1$) and for all real hyperplanes $H\subset \mathbb{P}^{g-1}$
containing $\langle P_2, \cdots, P_{g-1} \rangle$ one has $H$
intersects $C_1$ transversally at 2 points. In the proof we are
going to use this (at first sight stronger) statement.

\begin{proof}[Proof of Proposition \ref{Proposition1}]
We are going to prove for all $g\geq 3$ the existence of a
canonically embedded smooth real curve $X \subset \mathbb{P}^{g-1}$
of genus $g$ such that $X(\mathbb{R})$ has $g-1$ connected
components $C_1, \cdots, C_{g-1}$ and satisfying the following two
properties

\begin{enumerate}
\item[(P1)] For all $P_i \in C_i$ ($1 \leq i \leq g-1$) one has $\dim
\left( \langle P_1, \cdots , P_{g-1} \rangle \right) =g-2$.
\item[(P2)] For all $P_i \in C_i$ ($2 \leq i \leq g-1$) each hyperplane
$H \subset \mathbb{P}^{g-1}$ containing $\langle P_2, \cdots ,
P_{g-1} \rangle $ intersects $C_1$ transversally at two points.
\end{enumerate}

In the first part of the proof, we prove the existence of $X$ for
the (already known) case $g=3$. The arguments used to prove this
case will be generalized in the second part of the proof in order to
obtain a proof by induction on $g$. In both parts of the proof we
are going to use the following fact. Let $\Gamma _0$ be a
canonically embedded non-hyperelliptic real singular curve having an
isolated real node $S$ as its only singularity and such that $\Gamma
_0(\mathbb{R})\setminus \{ S \} $ has $n$ connected components.
There exists a real algebraic deformation $\pi :
\mathfrak{X}\rightarrow I$ with $I$ a small neighborhood of $0$ in
$[0, +\infty[ \subset \mathbb{R}$ such that $\pi ^{-1}(0)=\Gamma _0$
and for $t>0$ the curve $X_t=\pi^{-1}(t)$ is a smooth real complete
curve of genus $g$ such that $X_t(\mathbb{R})$ has $n+1$ connected
components (see e.g. \cite{ref4}*{Section 7}, it can be shown
directly by using part of Construction II in \cite{ref7}). We can
assume for all $t\in I$ the curve $X_t$ is not hyperelliptic. Using
the relative dualizing sheaf for this deformation we can assume it
is a family of canonically embedded real curves in
$\mathbb{P}^{g-1}$.

\begin{proof}[First part of the proof] Let $X_0$ be a real hyperelliptic curve of genus 2. It has a
unique real component $C_{0,1}$ and $C_{0,1}$ dominates
$\mathbb{P}^1(\mathbb{R})$ for the hyperelliptic covering (see
\cite{ref3}*{Section 6}). Take $Q+\overline{Q}$ general on $X_0$
(hence $Q \in X_0(\mathbb{C})\setminus X_0(\mathbb{R})$) and
consider the real linear system $| K_{X_0}+(Q+\overline{Q})|$ on
$X_0$. Since all real divisors in $g^1_2$ on $X_0$ consist of 2 real
points we have $Q+\overline{Q} \notin g^1_2$.

In both parts of the proof we use the following general fact
concerning smooth complex curves $M$ of genus $g\geq 2$. Let $P$ and
$Q$ be two different points on $M$ with $\dim |P+Q |=0$ (this is
always the case if $M$ is not hyperelliptic) and consider the linear
system $|K_M +P+Q|$. This is a base point free linear system on $M$
and it defines a morphism $\phi : M \rightarrow \mathbb{P}^g$ such
that the image $\Gamma \subset \mathbb{P}^g$ of $M$ is the nodal
curve of arithmetic genus $g+1$ obtained from $M$ by identifying $P$
and $Q$ to become an ordinary node $S=\phi (P)=\phi (Q)$ of $\Gamma$
and $\Gamma$ is embedded by the dualizing sheaf $\omega _{\Gamma}$
(this is well-known, an argument can be found in \cite{ref11}*{Lemma
5}).

Applying this argument using $|K_{X_0}+(Q+\overline{Q})|$ we obtain
a canonically embedded real singular curve $\Gamma _0\subset
\mathbb{P}^2$ of degree 4, birationally equivalent to $X_0$. The
singular point $S$ on $\Gamma _0$ is an isolated point on $\Gamma_0
(\mathbb{R})$ and projection with center $S$ on a real line
$\mathbb{P}^1 \subset \mathbb{P}^2$ induces a real covering $X_0
\rightarrow \mathbb{P}^1$ corresponding to the $g^1_2$ on $X_0$. The
real locus $X_0 (\mathbb{R})$ corresponds to the unique connected
component $C_{0,1}$ of $\Gamma _0(\mathbb{R})\setminus \{  S \} $.
Since $S \notin C_{0,1}$ one has\\

(P1') For all $P\in C_{0,1}$ one has $\dim \langle P, S \rangle
=1$.\\

Moreover, if $H\subset \mathbb{P}^2$ is a real line containing $S$
then $H$ induces a divisor on $X_0$ belonging to
$g^1_2+Q+\overline{Q}$. This divisor is real hence it contains two
different points of $X_0(\mathbb{R})$. On $\Gamma _0$ one has\\

(P2') Each real line $H\subset \mathbb{P}^2$ with $S\in H$
intersects $C_{0,1}$ transversally at 2 points.\\

We obtain a real family $\pi :\mathfrak{X}\rightarrow I \subset [0,
+\infty [ \subset \mathbb{R}$ of canonically embedded real curves of
genus 3 in $\mathbb{P}^2$ such that $\pi ^{-1}(0)=\Gamma _0$ and for
$t>0$ the curve $\pi ^{-1}(t)=X_t$ is smooth such that
$X_t(\mathbb{R})$ has 2 connected components. Let $C_{t,1}$ be the
connected component of $X_t(\mathbb{R})$ specializing to $C_{0,1}$
and let $C_{t,2}$ be the connected component of $X_t(\mathbb{R})$
specializing to $\{ S \}$. Let $\mathcal{C}_i$ be the union of those
components $C_{t,i}$ (including $S$ in case $i=2$). For the
classical topology on $\mathfrak{X}(\mathbb{C})$ those are closed
subsets. Consider the fibered product $\mathcal{C}_1 \times _I
\mathcal{C}_2$ and its subset $\mathcal{Z}$ defined by $(P_1,
P_2)\in \mathcal{Z}$ if and only if $\dim \langle P_1, P_2 \rangle
=0$ (i.e. $P_1=P_2$). This is a closed subset in $\mathcal{C}_1
\times _I \mathcal{C}_2$ and since the natural map $\mathcal{C}_1
\times _I \mathcal{C}_2 \rightarrow I$ is proper it follows the
image $Z$ of $\mathcal{Z}$ in $I$ is closed. Because of (P1') one
has $0 \notin Z$. Shrinking $I$ we can assume
$\mathcal{Z}=\emptyset$. Let $G_{\mathbb{R}}$ be the Grassmannian of
real lines in $\mathbb{P}^2$ and define $\mathcal{I}\subset
\mathcal{C}_2 \times G_{\mathbb{R}}$ by $(P,L)\in \mathcal{I}$ if
and only if $P\in L$. Let $\mathcal{Z}'\subset \mathcal{I}$ be
defined by $(P,L)\in \mathcal{Z}'$ if and only if $L$ does not
intersect $C_{\pi (P),1}$ transversally. Since $\mathcal{Z}'\subset
\mathcal{C}_2 \times G_{\mathbb{R}}$ is closed and the induced map
$\mathcal{C}_2 \times G_{\mathbb{R}} \rightarrow I$ is proper it
follows the image $Z'$ of $\mathcal{Z}'$ in $I$ is closed. Because
of (P2') one has $0 \notin Z'$. Shrinking $I$ we can assume
$\mathcal{Z}'=\emptyset$.

Take $t_0 \neq 0$ and let $X=X_{t_0}\subset \mathbb{P}^2$. It is a
canonically embedded real curve of genus 3 and $X(\mathbb{R})$ has
two connected components $C_i=C_{t_0,i}$ ($i=1,2$). Let $P_i\in C_i$
for $i=1,2$ then $(P_1,P_2)\notin \mathcal{Z}=\emptyset$, hence
$\dim \langle P_1, P_2 \rangle =1$. This implies (P1) for this curve
$X$. Let $P_2\in C_2$ and let $L$ be a real line in $\mathbb{P}^2$
with $P_2\in L$. Then $(P_2,L)\in \mathcal{I}$. Choose a family
$(P_{t,2}, L_t)_{t\geq 0}$ in $\mathcal{I}$ with
$(P_{t_0},L_{t_0})=(P_2,L)$. Then $P_{0,2}=S$ hence $L_0$ intersects
$C_{0,1}$ transversally at 2 points. Since $\mathcal{Z}'=\emptyset$
it follows all intersection of $L_t$ and $C_{1,t}$ ($t\geq 0$) is
transversal. Since $\bigcup _{t\geq 0}\{t \}\times L_t$ and
$\mathcal{C}_1$ are closed in the classical topology of $I\times
\mathbb{P}^2$ it follows $L$ intersects $C_1$ transversally at 2
points. This implies (P2) for this curve $X$.
\renewcommand{\qedsymbol}{}
\end{proof}

\begin{proof}[Second part of the proof] Repeating the arguments of the first part of the proof we are going to finish the proof by
induction on the genus. Assume $X_0\subset \mathbb{P}^{g-1}$ is a
canonically embedded smooth real curve of some genus $g\geq 3$
satisfying properties (P1) and (P1). Take $Q+\overline{Q}$ general
on $X_0$ (by assumption already $X_0$ is not hyperelliptic hence
$\dim |Q+\overline{Q}|=0$). Using $|K_{X_0}+(Q+\overline{Q})|$,
which is a real linear system on $X_0$, we obtain the canonically
embedded real singular curve $\Gamma _0 \subset \mathbb{P}^g$ having
a unique singular point $S$. This singular point is an isolated
point on $\Gamma _0(\mathbb{R})$. Choosing a real hyperplane
$\mathbb{P}^{g-1} \subset \mathbb{P}^g$ not containing $S$ then
projection with center $S$ on $\mathbb{P}^{g-1}$ induces a canonical
embedding $X_0\subset \mathbb{P}^{g-1}$ defined over $\mathbb{R}$.
Let $C_{0,i}$ ($1\leq i\leq g-1$) be the connected component of
$\Gamma _0(\mathbb{R})\setminus \{ S\}$ corresponding to the
component $C_i$ of $X_0(\mathbb{R})$. As before assumptions (P1) and
(P2) imply\\

(P1') For each $P_{0,i}\in C_{0,i}$ ($1\leq i\leq g-1$) one has
$\dim \left( \langle P_{0,1}, \cdots, P_{0,g-1} , S \rangle \right)
=g-1$.

(P2') For each $P_{0,i}\in C_{0,i}$ ($2\leq i\leq g-1$) each real
hyperplane $H$ in $\mathbb{P}^g$ containing $\langle P_{0,2},
\cdots, P_{0,g-1}, S \rangle$ intersects $C_{0,1}$ transversally at
two points.\\

Consider a real deformation $\pi : \mathfrak{X} \subset I\times
\mathbb{P}^g \rightarrow I \subset [0, +\infty [ \subset \mathbb{R}$
of canonically embedded real curves of genus $g+1$ with $\pi
^{-1}(0)=\Gamma _0\subset \mathbb{P}^g$ and for $t\neq 0$ one has
$X_t=\pi ^{-1}(t)$ is a smooth real curve of genus $g+1$ such that
$X_t(\mathbb{R})$ has $g$ connected components. For $1\leq i\leq
g-1$ and $t\neq 0$ let $C_{t,i}$ be the component specializing to
$C_{0,i}$ and let $C_{t,g}$ be the component specializing to $\{
S\}$. For $1\leq i\leq g$ let $\mathcal{C}_i$ be the union of those
components $C_{t,i}$ (including $S$ in case $i=g$). Let $\prod
_{i=1,I}^g \mathcal{C}_i$ be the set of $g$-uples $(P_1, \cdots,
P_g)$ with $P_i\in \mathcal{C}_i$ and $\pi (P_i)=\pi (P_j)$ for
$i\neq j$ and let $\mathcal{Z} \subset \prod _{i=1,I}^g
\mathcal{C}_i$ be defined by $(P_1, \cdots, P_g) \in \mathcal{Z}$ if
and only if $\dim \left( \langle P_1, \cdots, P_g \rangle \right) <
g-1$. Let $\mathcal{I} \subset \prod _{i=2,I}^g \mathcal{C}_i \times
G_{\mathbb{R}}$ (now $G_{\mathbb{R}}$ is the Grassmannian of real
linear subspaces of dimension $g-2$ in $\mathbb{P}^g$) be defined by
$(P_2, \cdots, P_g, H)\in \mathcal{I}$ if and only if $P_i\in
\mathcal{C}_i$, $\pi (P_i)=\pi (P_j)$ for $i\neq j$ and $P_i\in H$
and let $\mathcal{Z}' \subset \mathcal{I}$ be defined by $(P_2,
\cdots, P_g,H)\in \mathcal{Z}'$ if and only if $H$ does not
intersect $C_{t,1}$ transversally ($t=\pi (P_i)$). From (P1') and
(P2') it follows, by shrinking $I$, we can assume $\mathcal{Z}$ and
$\mathcal{Z}'$ being empty. Then taking $t_0 \neq 0$ and
$X=X_{t_0}\subset \mathbb{P}^g$ we obtain a canonically embedded
smooth real curve $X$ of genus $g$ such that $X(\mathbb{R})$ has $g$
connected components $C_i=C_{t_0,i}$. As in the previous case the
arguments imply this curve $X$ satisfies (P1) and (P2).
\renewcommand{\qedsymbol}{}
\end{proof}
\end{proof}

Condition 1 in Proposition \ref{Proposition1} implies for $P_i\in
C_i$ ($1\leq i\leq g-1$) one has $\dim |P_1+ \cdots +P_{g-1}|=0$.
This implies $\sepgon (X)\neq g-1$, hence we proved the existence of
separating ($M-1$)-curves of special type of separating gonality
$g$. As mentioned in the introduction we are going to prove that in
case $t\in T_{g,s}$ corresponds to a curve $X_t$ with separating
gonality $g-1$ then $t$ is not an inner point of $T_{g,s}$. This
indicates that it is natural to include the use the separating
gonality in the deformation argument used in the proof of
Proposition \ref{Proposition1} (i.e. to use condition 1 to prove
Theorem \ref{SpecialType}).

\section{The relation between special type and the separating
gonality}\label{section4}

We start by proving the following remarkable fact concerning
separating morphisms of degree $g-1$ on separating ($M-2$)-curves of
special type.

\begin{proposition}\label{Lemma1}
Let $X$ be a real separating (M-2)-curve of special type of genus $g
\geq 3$ satisfying $\sepgon (X)=g-1$, then each $g^1_{g-1}$ on $X$
having odd degree on each component of $X(\mathbb{R})$ is
half-canonical. In particular $X$ has only finitely many linear
systems $g^1_{g-1}$ associated to separated morphisms of degree
$g-1$.
\end{proposition}

\begin{proof}
We assume $X$ is canonically embedded in $\mathbb{P}^{g-1}$ (as a
matter of fact $X$ cannot be hyperelliptic (see \cite{ref3}*{Section
6}) and for an effective divisor $E$ on $X$ we write $\langle E
\rangle$ to denote its linear span in $\mathbb{P}^{g-1}$. Let $C_1,
\cdots, C_{g-1}$ be the connected components of $X(\mathbb{R})$ and
assume for each covering $f:X\rightarrow \mathbb{P}^1$ of degree $g$
having degree 1 on $C_i$ for $2\leq i\leq g-1$ one has
$f(C_1)=\mathbb{P}^1(\mathbb{R})$. Let $h:X\rightarrow \mathbb{P}^1$
be a separating morphism of degree $g-1$ and let $E$ be a real fiber
of $h$ (hence $E=Q_1 + \cdots Q_{g-1}$ for $Q_i\in C_i$ for $1\leq
i\leq g-1$). Because of the Riemann-Roch Theorem $\dim |K_X -E| \geq
1$ and $|K_X - E|$ has odd parity on each $C_i$ ($1\leq i\leq g-1$).
Since $\deg (K_X-E)=g-1$ each real divisor of $|K_X - E|$ is again
of type $Q_1 + \cdots Q_{g-1}$ with $Q_i\in C_i$ for $1\leq i\leq
g-1$. Choose $P_1\in C_1$ and let $P_1+ P_2+ \cdots + P_{g-1}$ be a
real fiber of $h$ and $P_1+Q_2+ \cdots +Q_{g-1}\in |K_X-E|$ (here
$P_i, Q_i\in C_i$ for $2\leq i\leq g-1$). In case $P_1+ P_2+ \cdots
+P_{g-1}\neq P_1+Q_2+ \cdots + Q_{g-1}$ we can assume without loss
of generality that $P_{g-1}\neq Q_{g-1}$. Assume $X$ is canonically
embedded and assume $Q_{g-1}\in \langle P_1+ \cdots P_{g-2}\rangle$.
Since $P_{g-1}\in \langle P_1+ \cdots P_{g-2}\rangle$ it follows
$\dim (\langle P_1+ \cdots P_{g-1}+Q_{g-1} \rangle) =g-3$, and
therefore $\dim (|P_1 + Q_2+ \cdots +Q_{g-2}|)=1$. Hence there would
exist a $g^1_{g-2}$ on $X$ having odd degree on $C_1, \cdots,
C_{g-2}$. Since $|P_1+Q_2+ \cdots + Q_{g-2}|$ has odd parity on each
$C_i$ for $1\leq i\leq g-2$ and because of the existence of one more
component $C_{g-1}$ this is impossible. This proves $Q_{g-1} \notin
\langle P_1+ \cdots +P_{g-2} \rangle$ and therefore $\dim |P_1+
\cdots +P_{g-2}+Q_{g-1}| =0$. Since $2P_1+P_2+ \cdots
+P_{g-2}+Q_{g-1}\in |K_X-(Q_2+\cdots +Q_{g-2}+P_{g-1})|$ one obtains
$\dim |2P_1+P_2+ \cdots + P_{g-2}+Q_{g-1}| =1$ and $P_1$ is not a
base point of $|2P_1+P_2+ \cdots +P_{g-2}+Q_{g-1}|$. A morphism $f:X
\rightarrow \mathbb{P}^1$ associated to the base point free linear
system defined by $|2P_1 + P_2 +  \cdots + P_{g-2} + Q_{g-1}|$ is
ramified at $P_1 \in C_1$ and there is no other point of $C_1$ at
that fiber. This implies the existence of a fiber containing no
point of $C_1$, hence $f(C_1) \neq \mathbb{P}^1(\mathbb{R})$ and
therefore the existence of a divisor $D\in |2P_1+P_2+\cdots
+P_{g-2}+Q_{g-1}|$ with $\Supp (D)\cap C_1=\emptyset$. In case the
linear system $|2P_1 + P_2 + \cdots +P_{g-2} + Q_{g-1}|$ has no base
point the morphism $f$ has degree $g$ and it has odd degree on $C_2,
\cdots, C_{g-1}$ and therefore $f(C_1) \neq \mathbb{P}^1
(\mathbb{R})$ contradicting our assumptions. We are going to show
that by deforming $(Q_2, \cdots, Q_{g-2}, P_{g-1})$ on $C_2 \times
\cdots \times C_{g-2} \times C_{g-1}$ we obtain such contradiction.

Consider the closed subset $Z \subset X^{(g)}(\mathbb{R}) \times C_2
\times \cdots \times C_{g-1}$ defined by $(D', Q'_2, \cdots ,
Q'_{g-2}, P'_{g-1})\in Z$ if and only if $D'\in |K_X-(Q'_2 + \cdots
+ Q'_{g-2} + P'_{g-1})|$. Consider the morphisms $p_1 : Z
\rightarrow C_2 \times \cdots \times C_{g-1}$ and $p_2 : Z
\rightarrow X^{(g)}(\mathbb{R})$ induced by projection. Since $\dim
(|Q'_2 + \cdots + Q'_{g-2} + P'_{g-1}|)=0$ for all $(Q'_2, \cdots,
Q'_{g-2}, P'_{g-1})\in C_2 \times \cdots \times C_{g-1}$, it follows
from the Riemann-Roch Theorem that $p_1^{-1}(Q'_1, \cdots ,Q'_{g-1},
P'_{g-1})\cong \mathbb{P}^1(\mathbb{R})$, in particular $p_1$ is a
locally trivial $\mathbb{P}^1(\mathbb{R})$-bundle. Let $d_0=(Q_2,
\cdots, Q_{g-2}, P_{g-1})$, we proved there exists $(d_0,D)\in
p_1^{-1}(d_0)$ such that $D\notin X(\mathbb{R})^{(g)}$. Since
$X(\mathbb{R})^{(g)}$ is closed in $X^{(g)}(\mathbb{R})$ there
exists a classical neighborhood $V$ of $D$ in $X^{(g)}(\mathbb{R})$
such that $V\cap X(\mathbb{R})^{(g)}=\emptyset$. Let $S=\{ d\in C_2
\times \cdots \times C_{g-1} : p_2 (p_1^{-1}(d))\cap V=\emptyset \}$
and assume $d_0 \in \overline{S}$. Take a neighborhood $U$ of $d_0$
in $C_2 \times \cdots \times C_{g-1}$, such that $p_1^{-1}(U)$ is
homeomorphic to $\mathbb{P}^1(\mathbb{R})\times U$ and $p_1 |
_{p_1^{-1}(U)}$ is identified with the projection
$\mathbb{P}^1(\mathbb{R})\times U\rightarrow U$. The closure of
$p_1^{-1}(S\cap U)$ in $X^{(g)} \times U$ is identified with
$\mathbb{P}^1(\mathbb{R}) \times (\overline{S\cap U})$ (here
$\overline{S\cap U}$ is the closure of $S\cap U$ in $U$) hence
$p_1^{-1}(d_0)$ belongs to the closure of $p_1^{-1}(S \times U)$.
But $p_2^{-1}(V)$ is a neighborhood of $(D,d_0)$ in $Z$ hence
$p_2^{-1}(V)\cap p_1^{-1}(S\cap U) \neq \emptyset$. Of course this
contradicts the definition of $S$, hence $d_0\notin \overline{S}$.
Hence there exists a neighborhood $U$ of $d_0$ in $C_2 \times \cdots
\times C_{g-1}$ such that for all $d'=(Q'_2, \cdots,
Q'_{g-2},P'_{g-1})\in U$ one has $p_2(p_1^{-1}(d'))\cap V\neq
\emptyset$, hence there exists a divisor $D'\in V$ with $D' \in |K_X
-(Q'_2 + \cdots + Q'_{g-2}+P'_{g-1})|$. In particular $|K_X -(Q'_2+
\cdots +Q'_{g-2}+P'_{g-1}|$ contain a divisor $D'$ containing a
non-real point in its support. Since $|K_X -(Q'_2 + \cdots +
Q'_{g-2} + P'_{g-1})|$ has odd parity on $C_2, \cdots, C_{g-1}$ and
even parity on $C_1$ it follows $\Supp (D')\cap C_1 = \emptyset$. In
case $|K_X -(Q'_2 + \cdots +Q'_{g-2} + P'_{g-1})|$ would contain a
base point for all $d'\in U$, using termminology from \cite{ref10},
it would imply $\dim ((W^1_{g-1}+W^0_1)(\mathbb{R}))\geq g-2$. Since
$W^1_{g-1}=g-4$ ($X$ is not hyperelliptic, so we can apply Martens'
Theorem, see \cite{ref10}) this is impossible. So we can assume
$|K_X-(Q'_2 + \cdots + Q'_{g-2} + P'_{g-1})|$ is base point free.
But then it corresponds to a covering $f':X\rightarrow \mathbb{P}^1$
of degree $g$ having odd degree on $C_2, \cdots, C_{g-1}$ and
$f'(C_1)\neq \mathbb{P}^1(\mathbb{R})$. This contradicts the
assumptions on $X$. This proves $|K_X - (P_1 + \cdots +
P_{g-1})|=|P_1 + \cdots + P_{g-1}|$and so $P_1 + \cdots + P_{g-1}$
is a half-canonical divisor. From parity considerations we also
obtain $\dim | P_1 + \cdots + P_{g-1} | < 2$ for such divisor,
implying the finiteness of linear systems $g^1_{g-1}$ associated to
separating morphisms on a real separating ($M-2$)-curve of special
type.
\end{proof}

In \cite{ref1}*{Example 3} it is noted that each separating
($M-2$)-curve of genus 3 is of special type. It follows from the
previous Proposition this is not the case for genus $g\geq 4$.

\begin{corollary}\label{corollary2}
Let $g$ be an integer at least 4. There exist real separating
($M-2$)-curves of genus $g$ not of special type.
\end{corollary}

\begin{proof}
From \cite{ref7} we know there exists a dividing (M-2)-curve $X$
such that $\sepgon(X)=g-1$. Assume $X$ is of special type. Let $\pi
: \mathcal{X} \rightarrow S$ be a suited family for $X$ and $s_0\in
S(\mathbb{R})$ with $X=\pi ^{-1}(s_0)$. Let $\pi _{g-1} :
\mathcal{H} \rightarrow S$ be the parameterspace parameterizing
morphisms of degree $g-1$ from fibers of $\pi$ to $\mathbb{P}^1$.
From deformation theory of Horikawa it follows $\mathcal{H}$ is
smooth of dimension $4g-4$. Such morphism corresponds to a linear
system $g^1_{g-1}$, let $\mathcal{H}_h(\mathbb{C})$ be the subset of
$\mathcal{H}(\mathbb{C})$ corresponding to half canonical linear
systems $g^1_{g-1}$. This is a closed subset of
$\mathcal{H}(\mathbb{C})$ of dimension $3g-1$ and it is invariant
under complex conjugation, so $\mathcal{H}_h(\mathbb{C})$ are the
complex points of a closed subset $\mathcal{H}_h\subset \mathcal{H}$
defined over $\mathbb{R}$ and we find $\dim
(\mathcal{H}_h(\mathbb{R}))\leq 3g-1$, in particular for each $f\in
\mathcal{H}_h(\mathbb{R})$ one has $U\cap
\mathcal{H}(\mathbb{R})\neq \mathcal{H}_h(\mathbb{R})$. By
assumption there exists $[f]\in \pi ^{-1}_{g-1}(s_0)(\mathbb{R})$
such that $f : X \rightarrow \mathbb{P}^1$ is a separating morphism.
From Proposition \ref{Lemma1} it follows $[f]\in
\mathcal{H}_h(\mathbb{R})$. Hence $f$ deforms to a separating
morphism $[f']$ that is not half-canonical. By Proposition
\ref{Lemma1} this is defined on a fiber $X'$ of $\pi$ not of special
type.
\end{proof}

The previous proof also implies the following fact.

\begin{corollary}\label{relatie1}
$T_{g,s}\cap T_{g,g-1}\subset \overline{T_{g,ns}}$ in case $g\geq
4$.
\end{corollary}

We now prove the strong relation between both partitions of $T_g$.

\begin{theorem}\label{Proposition3}
Let $X$ be a real separating ($M-2$)-curve of genus $g$ not of
special type, then $\sepgon (X)=g-1$, hence $T_{g,ns} \subset
T_{g,g-1}$.
\end{theorem}

\begin{proof}
Assume $X$ is a separating ($M-2$)-curve of genus $g$ not of special
type. We can assume that there exists a separating real morphism
$f:X\rightarrow \mathbb{P}^1$ of degree $g$, otherwise clearly
$\sepgon (X)=g-1$. For each component $C_i$ of $X(\mathbb{R})$ one
has $f|_{C_i} : C_i \rightarrow \mathbb{P}^1(\mathbb{R})$ is a
covering of some degree $d_i \geq 1$ and $\sum _{i=1}^{g-1}d_i = g$.
It follows $d_i=1$ except for one value $d_i=2$, we can assume
$d_1=2$ and $d_2= \cdots = d_{g-1}=1$. The morphism $f$ corresponds
to a linear systems $g=g^1_g$ and $|K_X -g^1_g|\neq \emptyset$.
Since $\deg (K_X-g^1_g)=g-2$ and $|K_X -g^1_g|$ has odd degree on
$C_2, \cdots, C_{g-1}$ it follows $|K_X -g^1_g|=\{ Q_2+ \cdots
+Q_{g-1} \}$ for some $Q_i\in C_i$. By assumption $X$ is not
special, hence there exists $Q'_i\in C_i$ for $2\leq i\leq g$ such
that $|K_X -(Q'_2+ \cdots + Q'_{g-1})|$ defines a $g^1_g=g'$ on $X$
such that $g'$ corresponds to a non-separating morphism $f':X
\rightarrow \mathbb{P}^1$, hence $C_i$ for $2\leq i\leq g-1$
dominates $\mathbb{P}^1(\mathbb{R})$ but $C_1$ doesn't. This implies
$g'$ contains a real divisor $P'+\overline{P'}+P'_2+ \cdots +
P'_{g-1}$ with $P'+\overline{P'}$ a non-real point of $X$. Take a
path $\gamma : [0, 1]\rightarrow C_2 \times \cdots \times C_{g-1}$
with $\gamma (0)=(Q_2, \cdots, Q_{g-1})$ and $\gamma (1)=(Q'_2,
\cdots, Q'_{g-1})$. Let $\gamma (t)=(Q_2(t), \cdots, Q_{g-1}(t)$ and
$g^1_g(t)=|K_X-(Q_2(t)+ \cdots + Q_{g-1}(t))|$. In case $g^1_g(t)$
is base point free for all $t\in I$ we can find a family of real
morphisms $f_t:X \rightarrow \mathbb{P}^1$ with $f_0=f$ and
$f_1=f'$. Since the topological degree of $f$ (resp. $f'$) is $(2,1,
\cdots, 1)$ (resp. $(1, \cdots, 1, 0)$) and this discrete invariant
should be constant in this family, we obtain a contradiction. So
there exists $t_0\in I$ such that $g^1_g(t_0)$ has a base point.
Moreover for $t<t_0$ we can assume $g^1_g(t)$ defines a separating
morphism $f_t:X \rightarrow \mathbb{P}^1$. By continuity it follows
each divisor on $g^1_g(t_0)$ is of type
$\overline{P_1}+\overline{P'_1}+\overline{P_2}+\cdots
+\overline{P_{g-1}}$ with $\overline{P_1}, \overline{P'_1}\in C_1$
and $\overline{P_i}\in C_i$ for $2\leq i\leq g-1$. Assume
$\overline{P_2}$ is a fixed point of $g^1_g(t_0)$ then for
$\overline{P'_2}\in C_2 \setminus \{ \overline{P_2} \}$ there is no
divisor in $g^1_g(t_0)$ containing $\overline{P'_2}$, a
contradiction. So we find $\overline{P_i}$ is not a fixed point for
$2\leq i\leq g-1$, hence we can assume $\overline{P'_1}$ is a fixed
point. But then we find $\dim |\overline{P_1}+\overline{P_2}+ \cdots
+ \overline{P'_{g-1}}| =1$, hence $g^1_g(t_0)-\overline{P'_1}$
defines a separating morphism $f_0: X\rightarrow \mathbb{P}^1$ of
degree $g-1$. This proves $\sepgon (X)=g-1$.
\end{proof}

\begin{corollary}\label{ExistenceSpecialSepgon}
Let $g\geq 3$. There exist separating ($M-2$)-curves $X$ of special
type such that $\sepgon (X)=g-1$.
\end{corollary}

\begin{proof}
From Lemma \ref{SpecialTypeClosed} it follows $T_{g,ns}$ is an open
subset of $T_g$ and it follows from Corollary \ref{corollary2} that
$T_{g,ns}\neq \emptyset$. It is already proved in Theorem
\ref{SpecialType} that $T_{g,ns}\neq T_g$ (indeed, $T_{g,ns}\neq
\emptyset$). Since $T_g$ is connected it follows $T_{g,ns}$ is not
closed. On the other hand we just proved $T_{g,ns}\subset T_{g,g-1}$
and it is proved in Lemma \ref{SepgonClosed} that $T_{g,g-1}$ is
closed. Hence $T_{g,ns}\neq T_{g,g-1}$ and therefore $T_{g,s}\cap
T_{g,g-1}\neq \emptyset$.
\end{proof}

The proof of this corollary implies the following inclusion.

\begin{corollary}\label{relatie2}
$\overline{T_{g,ns}}\cap T_{g,s}\subset T_{g,g-1}$
\end{corollary}

Together with Corollary \ref{relatie1} This implies

\begin{corollary}\label{relatie3}
$\overline{T_{g,ns}}=T_{g,g-1}$.
\end{corollary}

\begin{corollary}\label{general}
Let $g\geq 4$. There exist general separating ($M-2$)-curves of
genus $g$ of special type and general separating ($M-2$)-curves of
non-special type.
\end{corollary}

\begin{proof}
From Corollary \ref{corollary2} it follows that there exist
separating ($M-2$)-curves of genus $g$ of non-special type. Then
from Lemma \ref{SpecialTypeClosed} we know there exist general
separating ($M-2$)-curves of genus $g$ of non-special type. In
\cite{ref11} it is proved that $T_{g,g}\neq \emptyset$ (this is also
obtained from Proposition \ref{Proposition1}). Such curve does not
belong to $\overline{T_{g,ns}}$ hence $T_g \setminus
\overline{T_{g,ns}}\subset T_{g,s}$ is an open non-empty subset of
$T_g$ and it parameterizes general separating ($M-2$)-curves of
special type.
\end{proof}

The previous result also implies the following remarkable corollary.

\begin{corollary}\label{corollary1}
There exist dividing (M-2)-curves $X$ of genus $g\geq 4$ such that
$W^1_{g-1}(X)(\mathbb{R})$ has an isolated point.
\end{corollary}

\begin{proof}
Again let $X$ be a dividing (M-2)-curve of special type having
separable gonality $g-1$. From Corollary
\ref{ExistenceSpecialSepgon} we know $X$ does exist. A separating
morphism $f:X \rightarrow \mathbb{P}^1$ of degree $g-1$ corresponds
to a complete base point free $g^1_{g-1}$ on $X$, hence it belongs
to a connected component of $W^1_{g-1}(X)(\mathbb{R})$ and each
$g'^1_{g-1}$ close to $g^1_{g-1}$ is also base point free, complete
and induces a separating morphism. But from Proposition \ref{Lemma1}
it follows $g'^1_{g-1}$ has to be half-canonical. Since a curve has
only finitely many half-canonical linear systems it follows
$g^1_{g-1}$ corresponds to an isolated point of
$W^1_{g-1}(X)(\mathbb{R})$.
\end{proof}

This corollary is in sharp contrast (in case $g\geq 5$) to the fact
that the dimension of each component of $W^1_{g-1}(X_{\mathbb{C}})$
is at least $g-4$. In the final remark we explain that it seems to
indicate difficulties in studying the real gonality of real curves.

\begin{remark}
In his paper \cite{ref9} E. Ballico considers an upper bound for the
real gonality of real curves. In moving families of real curves $X$
some components of $W^1_d(\mathbb{R})$ existing on general curves
can vanish at ''transition'' curves (meaning curves having such
components but not on all curves of some neighborhood in the moduli
space; this terminology is not used in loc. cit.). In his arguments
the author proves that having such a transition curve using degree
$[(g+3)/2]$ (this is the gonality of a general complex curve of
genus $g$) then there is a real pencil of degree at most
$[g+3)/2]+3$ that propagates to all nearby real curves of it. How to
finish the argument to conclude that it propagates on a dense set of
the moduli space of real curves is not clear to me (it seems to me
there is no argument in loc. cit.). As a matter of fact, the
previous corrolary shows that such components of $W^1_d(\mathbb{R})$
can vanish in isolated points at those transition curves. In
particular those transition curves do not need to have a singular
locus of $W^1_d(X_{\mathbb{C}})$ of dimension at least 1; the basic
tool in loc. cit. is the study of complex curves having a singular
locus of some $W^1_d$ of dimension at least one (or more). The
previous corollary is the most extreme case showing what could go
wrong in the argument from \cite{ref9}. On the other hand, it is
clear that the arguments coming from \cite{ref9} had much influence
on the present paper.
\end{remark}

\end{document}